\documentclass{amsproc}
\usepackage{amsmath}
\usepackage{amsfonts}
\usepackage{amssymb}
\usepackage{amsthm}
\usepackage[utf8]{inputenc}
\usepackage{breqn}
\usepackage{afterpage}
\usepackage{longtable}
\usepackage{indentfirst}
\usepackage{caption}
\usepackage{subcaption}
\usepackage{graphicx}
\graphicspath{ {./images/} }
\newtheorem{theorem}{Theorem}[section]
\newtheorem{lemma}[theorem]{Lemma}
\newtheorem{corollary}[theorem]{Corollary}
\newtheorem{proposition}[theorem]{Proposition}

\theoremstyle{definition}
\newtheorem{definition}[theorem]{Definition}

\theoremstyle{remark}
\newtheorem{remark}[theorem]{Remark}

\numberwithin{equation}{section}

\begin{document}

\title{A mean value criterion for plurisubharmonic functions}

\author{Karim Rakhimov}
\author{Shomurod Shopulatov}


\address{CNRS, Univ. Lille, UMR 8524 - Laboratoire Paul Painleve, F-59000 Lille, France}
\email{karimjon1705@gmail.com}
\address{V.I.Romanovskiy Institute of Mathematics AS RUz, Tashkent, Uzbekistan}
\email{shomurod\_shopulatov@mail.ru}
\curraddr{}

\email{}

\subjclass[2010]{31B05,31C10,32U15}
\date{}

\keywords{Plurisubharmonic functions; Blaschke-Privalov theorem; $p(t)$-subharmonic functions}


\keywords{}

\begin{abstract}
In this paper we prove a criterion  for plurisubharmonic functions in terms of integral mean by complex ellipsoids. Moreover, by using the  criterion we prove an analogue of Blaschke-Privalov theorem for plurisubharmonic functions. 
\end{abstract}

\maketitle

\section{Introduction}
It is well known that an upper semi-continuous function $u$ is subharmonic on $D\subset \mathbb{R}^n$ if the mean value over the small spheres or over the small balls at any point of $D$ is equal or greater then the value of the function at this point. Plurisubharmonic ($psh$) functions (in  $\mathbb{C}^n$) are defined by using subharmonicity at complex lines. In the first part of this work we introduce an integral criterion for $psh$ functions so that it can be a subharmonic-like definition for $psh$ functions. In order to state the first main result of the paper we need to introduce some notions. We consider in $ \mathbb {C}^{n} $ the following class of ellipsoids
$$ E(r_1,\ldots ,r_n) = \left\{ \frac {|z_1|^2} {r_1^2} +\cdots+ \frac {|z_n|^2} {r_n^2}  \leq 1 \right\}, $$
where $r_j>0$ for any $ 1 \le j \le n $. By taking $r_1=R$  and $r_j=r $ for all $ 2 \le j \le n $  in $ E(r_1,\ldots ,r_n) $ let us denote
$ E(R,r) :=E(R,r,\ldots,r) ,$ where $R$ an $r$ are positive numbers.

It is not difficult to check that for any unitary matrix $T$ and for any complex line $l \subset \mathbb{C}^n $  passing through origin  $(T\circ E(r_1,\ldots ,r_n))\cap l$ is a disc in $l$. Let $D\subset \mathbb{C}^n$ be a domain. For integrable function  $u$ on $D$ we consider the following mean value over $ E(r_1,\ldots, r_n) $
$$M_u(z^0, T, E(r_1,\ldots,r_n))= \frac {1} {V(r_1,\ldots,r_n)} \int\limits_{z^0 + T\circ E(r_1,\ldots ,r_n)} u(\xi) dV(\xi),$$
where  $V(r_1,\ldots ,r_n) = \int\limits_{E(r_1,\ldots,r_n)} dV(\xi) = \frac {\pi^n r_1^{2}\cdots r_n^{2}  } {n!} $ is the volume of $ E(r_1,\ldots,r_n) $.

The following theorem gives a criterion for $ psh $ functions in terms of $ E(r_1,\ldots,r_n) $  and $ E(R,r)$ ellipsoids. We prove that  $psh$  functions satisfy the mean value integral inequality in terms of $ E(r_1,\ldots,r_n) $  and $ E(R,r)$ ellipsoids and on the other hand upper semi-continuous function satisfying one of the inequalities is  $psh$ . The theorem states as follows 
 \begin{theorem}\label{intromeanvaluecriterion}
Let $D\subset \mathbb{C}^n$ be a domain and $u$ be an upper semi-continuous function on $D$. Then the following properties are equivalent:
\begin{enumerate}
    \item[a)] $u$ is $psh$ on $D$;
    \item[b)]  for any  unitary matrix $T$  with  $ T\circ E(r_1,\ldots ,r_n)\subset D$ the following inequality holds
$$  u(z^0) \leq M_u(z^0, T, E(r_1,\ldots ,r_n));$$

    \item[c)] for any unitary matrix $T$ there exists $r_0>0$ small enough such that for any $(R,r)$ with $\max\{R, r\}\le r_0$ the following inequality holds
$$
  u(z^0) \leq M_u(z^0, T, E(R,r));
$$

\end{enumerate}

 \end{theorem}

An interesting criterion for subharmonic functions is given by Blaschke and Privalov (see \cite{B69}). It asserts that an upper semi-continuous in the domain $D \subset \mathbb {R}^{n}$ function $u(x)$ with $u(x)\not\equiv -\infty,$  is subharmonic if and only if
$ \overline{\bigtriangleup} u(x)\geq 0 $ for all $ x^0\in D\setminus u_{-\infty}$ (see below Theorem \ref{B-P}). Here $u_{-\infty}:=\{x\in D:u(x)=-\infty\}$ and $ \overline{\bigtriangleup} u(x) $ is a generalised Laplace operator of the function $ u $ at the point $ x $ constructed by the mean over the spheres or by the mean of balls (see \cite{B16}, \cite{B69} \cite{P25}, \cite{P41}, \cite{P412}, \cite{SSh21}).
In the section 3 we give also an analogue of Blaschke-Privalov theorem for $psh$ functions (Theorem \ref{analogBP}). In the proof we essentially use Theorem \ref{ptBP} (Appendix \ref{appendixA}).

For this purpose we define the following
 $$ \overline{D}_T u(z^0)=\varlimsup_{R\to 0}\varlimsup_{r\to 0}\frac{M_u(z^0, T, E(R,r))-u(z^0)}{R^2} $$
 and let $ \overline{D} u(z^0)= \inf_{T}\overline{D}_T u(z^0)$ where the infimum taken all over the unitary matrices.
  \begin{theorem}\label{introBPpsh} An upper semi-continuous function $u$ in the domain $D \subset \mathbb {C}^{n}$ with $u(z)\not\equiv -\infty$,  is  $psh$   if and only if
$$ \overline{D} u(z)\geq 0 \mbox{ for all } z\in D\setminus u_{-\infty}. $$
 \end{theorem}

The paper is organised as follows. The main theorems \ref{intromeanvaluecriterion} and \ref{introBPpsh} we prove in sections \ref{integralmcriterion} and \ref{secBlashkePrivalov} respectively. In Appendix \ref{appendixA} we introduce the  notion of $p(t)$-subharmonic functions for some weight function $p(t)$ which is used to prove Theorem \ref{intromeanvaluecriterion} and Theorem \ref{introBPpsh}. In general we show that if  $p(t)\ge 0$ then this notion is equivalent to subharmonic functions.

\subsection*{Acknowledgments} We would like to express our deep gratitude to professor Azimbay Sadullaev for introducing us this theme and for useful advice during the work. The first author is currently supported by the Programme Investissement d’Avenir (I-SITE ULNE /ANR-16-IDEX-0004 ULNE and LabEx CEMPI /ANR-11-LABX-0007-01) managed by the Agence Nationale de la Recherche.

\section{Integral criterion for $psh$ functions}\label{integralmcriterion}
Let us recall the definition of $psh$ functions.
\begin{definition}
Let $ D\subset \mathbb {C}^{n}$ be a domain. An upper semi-continuous function $u:D\rightarrow [-\infty,\infty),$ is called \textit{plurisubharmonic} in $ D $ (shortly $ u(z) \in psh(D) $) if for any complex line $ l $ the function $ u|_{l} $ is subharmonic in $ l \cap D $.
\end{definition}

Now we show the construction of a new integral criterion for $psh$ functions. Let us state the first main proposition of this section  
\begin{proposition}\label{mshtomeanvalue}
 Let $D\subset \mathbb{C}^n$ be a domain and $u$ be a $psh$ function on $D$. Then for any $z^0\in D$ and for any  unitary matrix $T$  with  $ z^0+T\circ E(r_1,\ldots ,r_n)\subset D$ the following inequality holds
\begin{equation}\label{pshr1r2rn}
      u(z^0) \leq M_u(z^0, T, E(r_1,\ldots ,r_n)).
\end{equation}
\end{proposition}

Let us first prove the following lemma which is used in the proof of the proposition above. 

\begin{lemma}\label{monotone}
  Let $u$ be a  $psh$ on $D$. Then  $M_u(z^0, T, E(r_1,\ldots,r_n))$ is non-decreasing by $r_j$ for any $1 \le j \le n $, i.e. for $r_{j}' \le r_{j}''$ we have $$ M_u(z^0, T, E(r_1,\ldots,r_{j-1},r_j',r_{j+1},\ldots r_n))\le M_u(z^0, T, E(r_1,\ldots,r_j'',\ldots r_n)). $$
\end{lemma}
\begin{proof}
Without loss of generality we can assume $T=\mathrm{Id}$ and we prove the assertion for $r_1$. Note that, since $u$ is $psh$ it is subharmonic by $z_1$ for any fixed $'z:=(z_2,\ldots z_n)$. First of all, let us make the following denotions:
$$E_{n-1}('r)=\left\{\frac{|z_2|^2}{r_2^2} + \ldots+\frac{|z_{n}|^2}{r_n^2} \le 1\right\},$$  $$'E(r_1)=\left\{ |z_1|^2  \leq r_1^2 \Big( 1 - \frac{|z_2|^2}{r_2^2} - \ldots -\frac{|z_{n}|^2}{r_n^2}\Big)\right\}$$
and finally
$$v(r_1,'z)=\frac{1}{V'(r_1)}\int\limits_{'E(r_1)} u(z_1,'z)dV (z_1), $$
where $$V'(r_1):=\pi r_1^2 \left( 1 - \frac{|z_2|^2}{r_2^2} -\ldots -\frac{|z_{n}|^2}{r_n^2}\right)$$ is the area of $'E(r_1)$.
After using Fubini's theorem, we get:
$$
 M_u(z^0, Id, E(r_1,\ldots,r_n)) =$$ $$=\frac{1}{V(r_1,\ldots,r_n)}\int\limits_{E_{n-1}('r)}V'(r_1)dV('z) \frac{1}{V'(r_1)}\int\limits_{'E(r_1)} u(z_1,'z)dV (z_1) =
$$
$$  = \frac {n!} {\pi^{n-1}  r_2^2 \cdots r_n^2} \int\limits_{E_{n-1}('r)} {\left( 1 - \frac{|z_2|^2}{r_2^2} - \ldots- \frac{|z_{n}|^2}{r_n^2}\right)}v(r_1,'z)dV('z).$$
Since $u$ is subharmonic by $z_1$ the integral  $v(r_1,'z)$
is increasing  by $r_1$ i.e. $ v(r_1^*,'z) \le v(r_1^{**},'z)$ for all $ r_1^* \le r_1^{**} $.  Since $1 - \frac{|z_2|^2}{r_2^2} - \ldots- \frac{|z_{n}|^2}{r_n^2}$ is non-negative on $E_{n-1}('r)$ and $v(r_1,'z)$ is non-decreasing by $r_1$  the following integral
$$\frac {n!} {\pi^{n-1} r_2^2 \cdots r_n^2} \int\limits_{E_{n-1}('r)} {\left( 1 - \frac{|z_2|^2}{r_2^2} -\ldots -\frac{|z_{n}|^2}{r_n^2}\right)v(r_1,'z)}dV('z)$$ 
is non-decreasing by $r_1$. Consequently, $M_u(z^0, Id, E(r_1,\ldots,r_n))$ is non-decreasing by $r_1$. In this manner the proof could be given for any $ r_j $ where $ 1 \le j \le n $.
\end{proof}

Similarly, we can easily get the following 

\begin{corollary}
  Let $u$ be a  $psh$  function on $D$  then for $ T\circ E(R,r))$ the mean value $M_u(z^0, T, E(R,r))$ is non-decreasing by both $R$ and $r$.
\end{corollary}
Now we are ready to prove the Proposition \ref{mshtomeanvalue}.
\begin{proof}[{Proof of Proposition \ref{mshtomeanvalue}}] Assume $u$ is a  $psh$  function on $D$. Since $u(T^{-1}\circ z)$ is also $psh$ on $T\circ D$ it is enough to prove the assertion for $T=\mathrm{Id}$. First assume $r_1=\ldots=r_n=r.$ Then $M_u(z^0, \mathrm{Id}, E(r_1,\ldots ,r_n))$ is just the integral mean by $B(z^0,r)$. Hence, in this case \eqref{pshr1r2rn} is true because $u$ is also a subharmonic function.
Take now any vector $(r_1,\ldots,r_n)$ with $E(r_1,\ldots ,r_n)\subset D$. By Lemma \ref{monotone} the integral mean $M_u(z^0, T, E(r_1,\ldots ,r_n))$ is monotonically increasing by $r_j$ for any $1\le j\le n$.   Let $r:=\min\{r_1,\ldots ,r_n\}$. Since $M_u(z^0, T, E(r_1,\ldots ,r_n))$ is monotonically non-decreasing  we have $M_u(z^0, T, E(r,\ldots ,r))\le M_u(z^0, T, E(r_1,\ldots ,r_n))$. Consequently,
 $$ u(z^0)\le M_u(z^0, T, E(r,\ldots ,r))\le M_u(z^0, T, E(r_1,\ldots ,r_n)).$$
We are done. \end{proof}
Now we give the next main proposition of this section. It is the converse of the Proposition \ref{mshtomeanvalue} but in strong sense in the term of $E(R,r)$ ellipsoids.
\begin{proposition}\label{meanvaluetopsh}
Let $ D \subset \mathbb{C}^n $ be a domain and $u$ be an upper semi-continuous function on $ D $. If for any $z^0\in D$ and for any unitary matrix $ T $ the following
\begin{equation}\label{meanvaluemsh}
    \exists r_0>0 : u(z^0) \leq M_u(z^0, T, E(R,r)), \forall R,r : \max\{R, r\}\le r_0, T\circ E(R,r) \subset D
\end{equation}
is true then $ u(z) \in psh(D) $.
\end{proposition}
\begin{proof} We fix a point $ z^0 $, say $ z^0 = 0 $, and a line $l\ni 0$. It is not difficult to see that there exists an unitary matrix $T$ such that $T\circ l=\{z\in \mathbb{C}^n: z_{2 }= z_{3} = \ldots = z_n=0\} $, so that without loss of generality we can assume $ l=\{z\in \mathbb{C}^n: z_{2 }= z_{3} = \ldots = z_n=0\} $. We apply the formula \eqref{meanvaluemsh} for ellipsoid $E(R,r)$ and for $\max\{R,r\}$ small enough so that
$$ u(0) \leq \frac {n!} {\pi^n R^2 r^{2n-2}} \int\limits_{E(R,r)} u(z)dV(z), $$
and by setting $'z=(z_2,\ldots ,z_{n})$ and using Fubini's theorem we have
\begin{equation}\label{meanvalue2}
u(0) \leq \frac {n!} {\pi^n R^2 r^{2n-2}} \int\limits_{B_1(R)} dV(z_1) \int\limits_{'E} u(z)dV('z),
\end{equation}
where  $B_1(R)=\{|z_1| \le R\}$ and $$'E=\left\{\frac {|z_2|^2+\ldots+|z_n|^2} {r^2} \leq 1 - \frac {|z_1|^2} {R^2}\right\}.$$

Now we evaluate the right side of \eqref{meanvalue2}. Since the function $ u(z) $ is upper semi-continuous, there exists a monotonically decreasing sequence of continuous functions $ u_j(z) $ such that $ u_j(z) \downarrow u(z) $. Now we fix $ j \in \mathbb N$ and $ \varepsilon > 0 $. Take an open set $$ O_j^\varepsilon(z) = \{ u_j(z) < u_j(z_1,0,\ldots,0) + \varepsilon \}. $$ Then it is easy to see that $ O_j^\varepsilon \supset l\cap D $. Fix any $R<r_0$ with $B_1(R)\subset l\cap D$. Then for any $\varepsilon>0$ there exists $r$ such that $E(R,r)\subset O_j^\varepsilon$. Hence for such $r$s we get:
\begin{align*}
u(0) \le &\frac {n!} {\pi^n R^2 r^{2n-2}} \int\limits_{B_1(R)} dV(z_1) \int\limits_{'E} u(z)dV('z) \le \\
\le   & \frac {n!} {\pi^n R^2 r^{2n-2}} \int\limits_{B_1(R)} dV(z_1) \int\limits_{'E} u_j(z)dV('z) < \\
   < & \frac {n!} {\pi^n R^2 r^{2n-2}} \int\limits_{B_1(R)} dV(z_1) \int\limits_{'E} (u_j(z_1,0,\ldots,0) + \varepsilon)dV('z) = \\
  =  & \frac {n!} {\pi^n R^2 r^{2n-2}} \int\limits_{B_1(R)}u_j(z_1,0,\ldots,0) dV(z_1) \int\limits_{'E} dV('z)+\varepsilon = \\
   = & \frac {n} {\pi R^2} \int\limits_{B_1(R)}\left(1 - \frac {|z_1|^2} {R^2}\right)^{n-1}u_j(z_1,0,\ldots,0) dV(z_1) +\varepsilon 
\end{align*}
By letting $j\to \infty$ and using the monotone converges theorem of Beppo Levi then tending $\varepsilon$ to $ 0 $  we get the following inequality:
$$u(0)\le \frac {n} {\pi R^2 } \int\limits_{|z_1| \le R}\left(1 - \frac {|z_1|^2} {R^2}\right)^{n-1}u(z_1,0,\ldots,0) dV(z_1), \mbox{ for any } R<r_0.$$
We can prove similar inequality at arbitrary point $z^0\in l\cap D$. Consequently, by Corollary \ref{cormeanvalue0} we can see that  $ u(z) $ is a subharmonic function on $ l\cap D $. Hence $ u(z) $ is a  $psh$  function on $D.$
\end{proof}
Finally, we have a revised version of  Theorem \ref{intromeanvaluecriterion}
\begin{theorem}\label{meanvaluecriterion}
Let $D\subset \mathbb{C}^n$ be a domain and $u$ be an upper semi-continuous function on $D$. Then the following properties are equivalent:
\begin{enumerate}
    \item[a)] $u$ is $psh$ on $D$;
    \item[b)] for any $z^0\in D$ and any  unitary matrix $T$  with  $ z^0+T\circ E(r_1,\ldots ,r_n)\subset D$ the following inequality holds
$$  u(z^0) \leq M_u(z^0, T, E(r_1,\ldots ,r_n));$$
\item[c)] for any $z^0\in D$ and any   unitary matrix $T$ there exists a small $r_0$ such that for any tuple $(r_1,\ldots ,r_n)$ with $\max\{r_1,\ldots ,r_n\}\le r_0$ the following inequality holds
$$  u(z^0) \leq M_u(z^0, T, E(r_1,\ldots ,r_n));$$
    \item[d)] for any $z^0\in D$ and any  unitary matrix $T$ there exists $r_0$ small enough such that for any $(R,r)$ with $\max\{R, r\}\le r_0$ the following inequality holds
$$
  u(z^0) \leq M_u(z^0, T, E(R,r));
$$
    \item[e)]
    for any $z^0\in D$ and any unitary matrix $T$  with  $ z^0+T\circ E(R,r)\subset D$ the following inequality holds
$$  u(z^0) \leq M_u(z^0, T, E(R,r)).$$

\end{enumerate}

 \end{theorem}
 
\begin{proof}
 By Proposition \ref{mshtomeanvalue} we have $a)\implies b)$. The implications $b)\implies c)\implies d)$ and  $b)\implies e)\implies d)$  are obvious. By Proposition \ref{meanvaluetopsh} we have $d)\implies a)$.
\end{proof}

 \section{Analogue of Blaschke-Privalov's theorem for $psh$ functions}\label{secBlashkePrivalov}

Before stating the main theorem of this section we remind the upper generalised Laplace operator and we recall the classical Blaschke-Privalov's theorem. Let $ u$  be an upper semi-continuous function in the domain $ D \subset \mathbb {R}^{n}$ and $u(x)\not\equiv -\infty$. For a point $ x^0 \in D \setminus u_{-\infty}$, where $u_{-\infty}:=\{x\in D:u(x)=-\infty\},$ we define the \emph{upper} $ \overline{\bigtriangleup} u(x^0) $ \emph{generalised Laplace operator} of the function $ u $ at the point $ x^0 $, constructed by the mean of spheres (or balls) with the following equality:
\begin{equation}\label{sphere}
\overline{\bigtriangleup} u(x^0) :=2n \cdot \varlimsup_{r\rightarrow +0} \frac { m_u (x^0,r)-u(x^0)} {r^2}.
\end{equation}
where $m_u (x^0,r)$ is the integral mean of $u$ on the sphere $S(x^0,r)$ (or ball $ B(x^0,r) $, for more information see Appendix \ref{appendixA}).
 
 \begin{theorem}[{Blaschke-Privalov}]\label{B-P} An upper semi-continuous in the domain $D \subset \mathbb {R}^{n}$ function $u(x)$ with $u(x)\not\equiv -\infty,$  is subharmonic  if and only if
 \begin{equation}\label{upperD}
      \overline{\bigtriangleup} u(x)\geq 0  \mbox{ for all } x^0\in D\setminus u_{-\infty}.
 \end{equation}

 \end{theorem}

 \subsection{Blaschke-Privalov theorem for $psh$ functions}
 Now we define 
 $$ \overline{D}_T u(z^0)=\varlimsup_{R\to 0}\varlimsup_{r\to 0}\frac{M_u(z^0, T, E(R,r))-u(z^0)}{R^2} $$
 and let
 $$ \overline{D} u(z^0)= \inf_{T}\overline{D}_T u(z^0)$$
 where the infimum is taken all over the unitary matrices $T$.
  \begin{theorem}\label{analogBP} An upper semi-continuous function $u$ in the domain $D \subset \mathbb {C}^{n}$ with $u(z)\not\equiv -\infty$,  is  $psh$   if and only if
$$ \overline{D} u(z)\geq 0  \mbox{ for all }  z^0\in D\setminus u_{-\infty}. $$
 \end{theorem}
\begin{proof}
By Proposition \ref{mshtomeanvalue} we can easily get the necessary condition. Now let the function $u$ is upper semi-continuous in the domain $ D \subset \mathbb{C}^n $ with $ u(z) \not\equiv -\infty $ such that $\overline{D} u(z)\geq 0 $ for all $ z \in D \setminus u_{-\infty} $. Let $l$ be a complex line. We shall show that $u|_l$ is subharmonic on $l\cap D$. Since in \eqref{upperD} the infimum is taken by all unitary matrices $T$ we can assume that $l=\{z_2=\ldots =z_n=0\}$. To prove subharmonicity of $u|_{l\cap D}$ we show that $\overline{\bigtriangleup}_p u|_{l\cap D}\ge 0$ for some weight function $p(t)\ge 0$ (see Theorem \ref{ptBP}).

Let us first work with the following difference 
$$M_u(z^0, Id, E(R,r))- M_{u(z_1,0,\ldots ,0)}(z^0, Id, E(R,r))=$$
$$=\frac {n!} {\pi^n R^{2} r^{2n-2}} \int\limits_{E(R,r)} ((u(z)-u(z_1,0,\ldots ,0))dV=  $$
$$=\frac {n!} {\pi^n R^{2} } \int\limits_{E(R,1)} (u(z_1,rz_2,\ldots ,rz_n)-u(z_1,0,\ldots ,0))dV  $$
Now we take $\varlimsup\limits_{r\to 0}$ and we have

$$\varlimsup_{r\to 0}\frac {n!} {\pi^n R^{2} } \int\limits_{E(R,1)} u(z_1,rz_2,\ldots ,rz_n)-u(z_1,0,\ldots ,0)dV  \le $$
$$\le \frac {n!} {\pi^n R^{2} } \int\limits_{E(R,1)} \varlimsup_{r\to 0}u(z_1,rz_2,\ldots ,rz_n)-u(z_1,0,\ldots ,0)dV =:I$$
Since $u$ is upper semi-continuous we have $\varlimsup\limits_{r\to 0}u(z_1,rz_2,\ldots ,rz_n)\le u(z_1,0,\ldots ,0)$ and so $I\le 0$. Hence
$$\varlimsup_{r\to 0} M_u(z^0, Id, E(R,r))\le \varlimsup_{r\to 0} M_{u(z_1,0,\ldots ,0)}(z^0, Id, E(R,r)).$$
Now we show that $ \varlimsup\limits_{r\to 0} M_{u(z_1,0,\ldots ,0)}(z^0, Id, E(R,r))=n^p_{u(z_1,0,\ldots ,0)}(z_1^0,R)$ (for definition of $n^p_u(z^0,r)$ see Appendix \ref{appendixA}).  Indeed,
$$\varlimsup_{r\to 0} M_{u(z_1,0,\ldots ,0)}(z^0, Id, E(R,r))=\frac {n!} {\pi^n R^{2} } \int\limits_{E(R,1)} u(z_1,0,\ldots ,0)dV = $$
$$=\frac {n} {\pi R^{2} }\int\limits_{|z_1|\le R}\left(1-\frac{|z_1|^2}{R^2}\right)^{n-1} u(z_1,0,\ldots ,0)dV(z_1)=n_u^p(z^0,R),$$
where $p(t)=n \cdot \left(1-t^2\right)^{n-1} t^{2n-1}$ (see the proof of Corollary \ref{cormeanvalue0}).
Finally, we have
$$\varlimsup_{r\to 0}\frac{M_u(z^0, T, E(R,r))-u(z^0)}{R^2}\le \frac { n_u^p (z^0,R)-u(z^0)} {R^2}.$$ 
Consequently,  $\overline{\bigtriangleup}_p u|_l\ge \overline{D}_{\mathrm{Id}} u$. Since $\overline{D}_{\mathrm{Id}} u \ge 0$ we have $\overline{\bigtriangleup}_p u|_\Pi\ge 0$.  Hence by Theorem \ref{ptBP} $u$ is subharmonic on $l\cap D$.
\end{proof} 

 \begin{remark} We can proof similar result in terms of $ E(r_1,\ldots ,r_n)$ ellipsoids. If we set
 $$ \overline{D}_T u(z^0)=\varlimsup_{r_1\to 0}\varlimsup_{r_2+\ldots +r_n\to 0}\frac{M_u(z^0, T, E(r_1,\ldots ,r_n))-u(z^0)}{r_1^2} $$
 or
 $$ \overline{D}_T u(z^0)=\varlimsup_{r_1\to 0}\ldots\varlimsup_{r_n\to 0}\frac{M_u(z^0, T, E(r_1,\ldots ,r_n))-u(z^0)}{r_1^2} $$
 and define $ \overline{D} u(z^0)= \inf_{T}\overline{D}_T u(z^0)$ as above using the same proof we will get the same result as Theorem \ref{analogBP}.  
 \end{remark}
 \begin{remark} Note that
  $\sup_{T}\overline{D}_T u\ge 0$ does not guaranty  plurisubharmonicity of $u$. Indeed, $u=|z_1|^2-|z_2|^2$ is not  $psh$  but we can show that $\overline{D}_{\mathrm{Id}} u\ge 0$. Before that we evaluate $ M_u(0, Id, E(R,r)) $.
  Assuming $z_1 = Rw_1, z_2 = rw_2$ we will have the following:
     $$ M_u(0, Id, E(R,r))= \frac{2}{\pi^2} \int\limits_{|w_1|^2 + |w_2|^2 \le 1 } (R^2|w_1|^2 - r^2|w_2|^2)dV'=
     $$ $$=\frac{2}{\pi^2} \int\limits_{|w_1|^2 + |w_2|^2 \le 1 } R^2|w_1|^2 dV'-\frac{2r^2}{\pi^2} \int\limits_{|w_1|^2 + |w_2|^2 \le 1 } |w_2|^2dV'$$
     So we have
     $\varlimsup\limits_{r\to 0}  M_u(0, Id, E(R,r))=cR^2 $,
     where $$c=\frac{2}{\pi^2} \int\limits_{|w_1|^2 + |w_2|^2 \le 1 } |w_1|^2 dV'>0.$$ Hence,
 $$ \overline{D}_{Id} u(0)=\varlimsup_{R\to 0}\varlimsup_{r\to 0}\frac{M_u(0, Id, E(R,r))-u(0)}{R^2} =c>0$$
  Similarly, we can show that  $\overline{D}_{\mathrm{Id}} u(z^0)> 0$ for any $z^0\in \mathbb{C}$. 

 \end{remark}
\appendix\label{appendix}
\section{$p(t)$-subharmonic functions}\label{appendixA}
In this section we introduce $ p(t)$-subharmonic functions in $ \mathbb{R}^n $ for continuous weight function $ p(t) $ with $ \int_{0}^{1}p(t)dt=1 $. In particular, we prove $p(t)$-subharmonic functions are subharmonic functions  for positive weight functions (see below Lemma \ref{frompsubtosub}).  We note that the auxiliary notion of $p(t)$-subharmonic function is very helpful in our work, since Corollary \ref{cormeanvalue0} is used in the proof of Proposition \ref{meanvaluetopsh} and the theorem \ref{ptBP} is used in the proof of  Theorem \ref{analogBP}.

\subsection{Relation between $p(t)$-subharmonic and subharmonic function}
Let $D\subset \mathbb{R}^n$ be a domain and $x^0\in D$. Let $p(t), 0\le t\le 1$ be a continuous function. Assume  $\int_{0}^{1}p(t)dt=1$. For $r>0$ set
\begin{equation}\label{npu}
  n_u^p(x^0,r):=\int_{0}^{1}p(t)m_u(x^0,rt)dt
\end{equation}
where $m_u(x^0,rt)$ is the mean value of $u$ on the sphere $S(z^0,rt)$ i.e.
$$ m_u(x^0,r): = \frac{1}{r^{n-1}\sigma_n} \int\limits_{S(x^0,r)} u(x)d\sigma.$$
\begin{remark}
If we choose $ p(t) $ in the following way $ p(t) = nt^{n-1}$, then $ n_u^p(z^0,r) = n_u(z^0,r) $, i.e. $ n_u^p(z^0,r) $ will equal to the mean value of the function $u$ over the ball $ B(z^0,r) $. 
\end{remark}

\begin{definition}\label{psubharmonic}
   Let $D\subset \mathbb{R}^n$. We say that $u$ is \emph{$p(t)$-subharmonic} on $D$ if it is upper semi-continuous on $D$ and for any $x^0\in D$  there exists $r_0(x^0)>0$ such that
   \begin{equation}\label{psubharmonicmean}
     u(x^0)\le  n_u^p(x^0,r) \mbox{ for any } r<r_0.
   \end{equation}
   If $u$ is $p(t)$-subharmonic at any point of $D$ then we say $u$ is $p(t)$-subharmonic on $D$.
   \end{definition}
\begin{theorem}[{see \cite[Theorem 4.12]{D97} or \cite{SSh21}}]
Let  $ D \subset \mathbb{R}^n $  be a domain. An upper semi-continuous function $ u:D \rightarrow [-\infty, \infty)$ is subharmonic if and only if 
 for every $ x^0 \in D $, there exists a sequence $ \{r_j\} $ decreasing to $ 0 $ such that  $u(x^0)\le  m_u(x^0,r_j)  $ for any $j$.
\end{theorem}
   \begin{lemma}\label{frompsubtosub}
    Assume $p(t)\ge 0$. Then $u$ is  $p(t)$-subharmonic if and only if it is subharmonic.
   \end{lemma}
   \begin{proof}
      Assume $u$ is a subharmonic function then there exists $r_0$ such that for any $0<t<r_0$ we have  $u(x^0)\le m_u(x^0,t)$. Since $p\ge 0$ we have \eqref{psubharmonicmean}.

      Let now $u$ be $p(t)$-subharmonic. It is enough to show that for any $x^0\in D$ there exists $r_j\searrow 0$ such that $u(x^0)\le m_u(x^0,r_j)$. Assume contrary, that for some point $x^0\in D$ there is not such a sequence. Then there exists $r_0>0$ such that $m_u(x^0,rt)< u(x^0)$ for any $r<r_0$. Since $p\ge 0$ and $\int_{0}^{1}p(t)dt=1$ we have  $n_u^p(x^0,r)< u(x^0)$. Contradiction. Hence $u$ is a subharmonic function.
   \end{proof}
   Similarly result can be given for harmonic functions. The following lemma for harmonic functions is true.
\begin{lemma}
  A continuous function $u$ on $D$ is harmonic if and only if for any $x^0\in D$ there exists $r_0>0$ such that for any $r<r_0$ the following holds
  \begin{equation}\label{pharmonic}
      n_u^p (x^0,r)=u(x^0).
  \end{equation}
 
\end{lemma}
\begin{proof}
If $u$ is harmonic then \eqref{pharmonic} is clear. Assume now $u$ satisfies \eqref{pharmonic}. Then by Lemma \ref{frompsubtosub} $u$ is a subharmonic function. Hence for any $x^0\in D$ we have $u(x^0)\le m_u(x^0,r)$. By definition of $n_u^p (x^0,r)$ we easily get $n_u (x^0,r)=u(x^0)$. Hence $u$ is a harmonic function.
\end{proof}
   The following corollary is used in the proof of Proposition \ref{meanvaluetopsh}.
\begin{corollary}\label{cormeanvalue0}
Let $u$ be an upper semi-continuous function in the domain $D\subset\mathbb{C}^k$ and $l$ be a natural number. Then $u$ is subharmonic on $D$ if and only if for any $z^0\in D$ there exists $r_0>0$ such that for any $r< r_0$ we have
$$u(z^0) \le \frac {(k+l)!} {\pi^{k} l! r^{2k}} \int\limits_{B(z^0, r)} {\left(1 - \frac {|z_{1}-z_1^0|^2  + \ldots+|z_{k}-z_k^0|^2 }{r^2}  \right)^l}u(z) dV  $$
where $B(z^0, r)$ is the ball centered at $z^0$ with radius $r$.
\end{corollary}
\begin{proof} Without loss of generality take $z^0=0$. Then by using Fubini's theorem we have
\begin{align*}
  u(0)\le & \frac{(k+l)!} {\pi^{k} l! r^{2k}} \int\limits_{B(r)} {\left(1 - \frac {|z_{1}|^2  + \ldots+|z_{k}|^2 }{r^2}  \right)^l}u(z) dV = \\
  =  & \frac{(k+l)!} {\pi^{k} l! r^{2k}} \int\limits_{0}^{r}d\tau \int\limits_{S(0,\tau)} \left(1 - \frac {|z_{1}|^2  + \ldots+|z_{k}|^2 }{r^2}\right)^l u(z) d\sigma = \\
  =&\frac{(k+l)!} {\pi^{k} l! r^{2k}}  \int\limits_{0}^{r} \left(1 - \frac {\tau^2} {r^2}\right)^{l} d\tau \int\limits_{S(0,\tau)} u(z) d\sigma = \\
 = & \frac{(k+l)!} {\pi^{k} l! r^{2k}} \int\limits_{0}^{r}\left(1 - \frac {\tau^2} {r^2}\right)^{l} \cdot \frac{\tau^{2k-1} \cdot 2 \pi^{k}} {(k-1)!} \cdot m_{u}(0,\tau) d\tau  
\end{align*}
By changing $\tau=rt$ we have
\begin{align*}
    u(0)\le & \frac {2(k+l)!} {l!(k-1)!} \int\limits_{0}^{1}\left(1 - t^2\right)^{l} \cdot t^{2k-1} m_{u}(0,rt)dt = n_{u}^p(0,r)
    \end{align*}
where $p(t):=\frac {2(k+l)!} {l!(k-1)! }\left(1 - {t^2} \right)^{l} \cdot t^{2k-1}$. Consequently, it is enough to show that $\int\limits_0^1p(t)dt=1$. Indeed, we have
$$\int\limits_0^1p(t)dt=\frac {(k+l)!} {l!(k-1)! }\int\limits_0^1 (1-x)^lx^{k-1}dx=\frac {(l+k)!} {l!(k-1)! }\cdot\frac{\Gamma(l+1)\Gamma(k)}{\Gamma(l+k+1)}=1.$$

\end{proof}

\subsection{Blaschke-Privalov theorem for $p(t)$-subharmonic functions}
Now we prove Blaschke-Privalov type theorem for $p(t)$-subharmonic functions. Like \eqref{sphere} we define the following operator
\begin{equation}
     \overline{\bigtriangleup}_p u (x^0) :=A \cdot \varlimsup_{r\rightarrow +0} \frac { n_u^p (x^0,r)-u(x^0)} {r^2}.
 \end{equation}
 where $A^{-1}=\frac{1}{2n} \int\limits_0^1 t^2 p(t)dt$. Note that if $p(t)\ge 0$ then $A>0$.
 
 First of all we shall show that the operator $ \overline{\bigtriangleup}_p$ is actually Laplace operator for $C^2$ functions.
\begin{lemma}
  If $u\in C^2(D)$ then 
  \begin{equation}\label{plaplaceequlap}
      \overline{\bigtriangleup}_p u=  \bigtriangleup u.
  \end{equation}
\end{lemma}
\begin{proof}
Without loss of generality we shall prove \eqref{plaplaceequlap} at $x=0$. As $u\in C^2(D)$ and assuming that $ 0 \in D $ there exists $r>0$ such that we have the following decomposition:
$$ u(x) = u(0) + \sum\limits_{i=1}^n x_i \frac{\partial u}{\partial x_i} (0) + \frac{1}{2!} \sum\limits_{i,j = 1}^n x_i x_j \frac{\partial^2 u}{\partial x_i \partial x_j} (0) + o(r^2t^2), $$
where $ r^2t^2 = \sum\limits_{i=1}^n x_i^2 $ and $0\le t\le 1$. After averaging both sides of the last equality by the sphere $ S(0,rt)$ where $ 0 < r < \rho(0,D) $ we have
\begin{align*}
    m_u(0,rt) - u(0) = & \frac{1}{2 \sigma_n (rt)^{n-1}}  \sum\limits_{i=1}^n \frac{\partial^2 u}{\partial x_i^2} (0) \int\limits_{S(0,rt)} x_i^2 d\sigma  + o(r^2t^2) \\
    =&\frac{1}{2n}  \sum\limits_{i=1}^n \frac{\partial^2 u}{\partial x_i^2} (0) r^2t^2  + o(r^2t^2)
\end{align*}
and then multiplying both sides to $ p(t) $ and integrating by $0\le t\le 1$ we get
\begin{align*}
  n_u^p(0,r) - u(0)=  & \triangle u(0) \cdot \frac{r^2}{2n} \int\limits_0^1 t^2 p(t)dt + \int_0^1 p(t)o(r^2t^2)dt
 \end{align*}
Note that $ \int_0^1 p(t)o(r^2t^2)dt=o(r^2)$ when $r\to 0$. So we have
 \begin{align*}
     n_u^p(0,r) - u(0)= &A^{-1} r^2\triangle u(0)   + o(r^2).
 \end{align*} 
Consequently, we have \eqref{plaplaceequlap}.
\end{proof}
Now we state similar result as Theorem \ref{B-P}
 \begin{theorem}\label{ptBP} Assume $p(t)\ge 0$. An upper semi-continuous in the domain $D \subset \mathbb {R}^{n}$ function $u(x)$, $u(x)\not\equiv -\infty,$  is $ p$-subharmonic if and only if
$$ \overline{\bigtriangleup}_p u(x)\geq 0 \quad \mbox{ for all }  x^0\in D\setminus u_{-\infty}. $$
 \end{theorem}
 The proof of the above theorem is similar to the proof of classical Blaschke-Privalov theorem but there are some technical changes. Therefore we will give the proof for reader's convenience.
 \begin{proof}[{Proof of the Theorem \ref{ptBP}}]
 The necessity of the theorem is clear. If $ u(x) $ is $p(t)$-subharmonic function in $ D $ then for all $  z^0 \in D $ we have $ u(z^0)\le n_u^p(z^0,r) \mbox{ for any } r<r_0 $, which is equivalently to $ \overline{\bigtriangleup}_p u(x)\geq 0$  for all $ x^0\in D\setminus u_{-\infty}$.
 
 In order to prove the sufficiency of the theorem take  any  ball $ B(x^0,r) \subset \subset D $ and any continuous function $\varphi$ on $  S(x^0,r)$ such that $u|_{\partial B}\le \varphi$.   It is well known that if for any $\varphi$ as above  the following inequality holds $ P[\varphi] - u \ge 0 $  where
 $$P[\varphi](x) = \int\limits_{S(x^0,r)} \varphi(y) \frac{r^2 - |x-x^0|^2}{\sigma_n r |x-y|^n} d\sigma(y),\ n \ge 2 $$ is the Poisson mean of $\varphi$, then $u$ is a subharmonic function (see for example \cite{S12}). Note that $P[\varphi]$ is a harmonic function.

The following function $ w(x) = \|x-x^0\|^2 -r^2  $ is subharmonic with laplacian $ \triangle w = 2n $ and it is identically equal to zero on the sphere $ S(x^0,r) $. Now we consider the following auxiliary function $ v = u -  P[\varphi] + \varepsilon w  $, on the ball $ B(x^0,r) $, it is upper semi-continuous and $ v < +\infty $. It is enough to show that $ v < 0 $ for any $ \varepsilon > 0 $. Assume by contrary. Then since $v$ is upper semi-coninuous  $ v $ must attain strictly finite positive maximum at the point $ x $ where $ u $ is finite. Note that we always have 
$$n_u^p(x^0,r)\le \sup\limits_{\overline {B(x^0,r)}} u(x). $$ Hence if $ v $ attains its maximum at the point $ x $  by definition we have $ \overline{\bigtriangleup}_p v(x)\le 0$. On the other hand at this point must hold the following relations:
$$ \overline{\bigtriangleup}_p v(x) = \varepsilon \cdot \triangle w + \overline{\bigtriangleup}_p u(x) = 2n\varepsilon + \overline{\bigtriangleup}_p u(x) > \overline{\bigtriangleup}_p u(x) \ge 0.$$
Consequently, $ \overline{\bigtriangleup}_p v(x) > 0 $. Contradiction.
  \end{proof}

\end{document}